\documentclass[12pt,leqno]{article}
\usepackage{amsthm,amsfonts,amssymb,amsmath,oldgerm}
\usepackage{epsfig}

\newcommand\R{\mathbb R}

\newcommand\br{\begin{rem}}
\newcommand\er{\end{rem}}
\newcommand\bp{\begin{pmatrix}}
\newcommand\ep{\end{pmatrix}}
\newcommand\be{\begin{equation}}
\newcommand\ee{\end{equation}}
\newcommand\ba{\begin{equation}\begin{aligned}}
\newcommand\ea{\end{aligned}\end{equation}}

\newcommand{\CalF}{\mathcal{F}}

\newcommand{\CalO}{\mathcal{O}}

\newcommand{\CalS}{\mathcal{S}}

\newcommand{\CC}{{\mathbb C}}
\newcommand{\BbbA}{{\mathbb A}}

\newcommand{\const}{\text{\rm constant}}

\newcommand{\Span}{{\rm span }  \,}
\newcommand{\Rank}{{\rm rank} \, }

\newtheorem{theo}{Theorem}[section]
\newtheorem{prop}[theo]{Proposition}
\newtheorem{cor}[theo]{Corollary}
\newtheorem{lem}[theo]{Lemma}
\newtheorem{defi}[theo]{Definition}

\newtheorem{rem}[theo]{Remark}

\numberwithin{equation}{section}
\title
{ Pointwise Green function bounds and long-time stability of
large-amplitude noncharacteristic
boundary layers }
\author{\sc \
Shantia Yarahmadian\thanks{Indiana University, Bloomington,;
syarahma@indiana.edu:
 }
and
Kevin Zumbrun\thanks{
Indiana University Department of Mathematics;
kzumbrun@indiana.edu:
 }}
\begin{document}
\maketitle
\begin{abstract}
Using pointwise semigroup techniques of Zumbrun--Howard and
Mascia--Zumbrun, we obtain sharp
global pointwise Green function bounds for noncharacteristic boundary
layers of arbitrary amplitude.
These estimates allow us to analyze linearized and
nonlinearized stability of noncharacteristic boundary layers of
one-dimensional systems of conservation laws, showing that both are
equivalent to a numerically checkable Evans function condition. Our
results extend to the large-amplitude case results obtained for small
amplitudes by Matsumura, Nishihara and others using energy
estimates.
\end{abstract}
\bigbreak
\section{Introduction}\label{introduction}
Boundary layers appear in many physical settings, such as gas
dynamics, MHD, and rotating fluids; see, for example, the physical
discussion in \cite{SGKO}. In this paper, we study the stability of
boundary layers assuming that the boundary layer solution is
noncharacteristic which means that signals are transmitted into or
out of but not along the boundary. Specifically, we onsider a
boundary layer, or stationary solution,
\begin{equation}\label{profile}
u=\bar{u}(x), \quad \lim_{z\to +\infty} \bar{u}(z)=u_+,
\quad \bar u(0)=u_0
\end{equation}
of a system of conservation laws on the quarter-plane
\begin{equation}\label{parabolic}
u_t +  f(u)_{x} = (B(u)u_{x})_{x}, \quad x,t>0,
\end{equation}
$u,f\in \mathbb{R}^n$, $B\in\mathbb{R}^{n \times n}$,
with initial data $u(x,0)=g(x)$ and Dirichlet boundary condition
\begin{equation}\label{cond}
u(0,t)= h(t).
\end{equation}
A fundamental question is whether or not
such boundary layer solutions are $\it stable$ in the sense of PDE,
i.e., whether or not a sufficiently small perturbation of $\bar{u}$
remains close to $\bar{u}$, or converges time-asymptotically to
$\bar{u}$, under the evolution of \eqref{parabolic}.

Long-time stability of boundary layers has been considered for
scalar equations in \cite{LN,LY}
and for the equations of isentropic gas dynamics  in \cite{MN, KNZ}.
The latter results, obtained by energy estimates, apply to arbitrary
amplitude layers of ``expansive inflow'' type analogous to
rarefaction waves, but only to small-amplitude layers of
``compressive inflow or outflow'' layers analogous to shock waves or
``expansive outflow'' type. For general symmetric
hyperbolic-parabolic systems, stability of small-amplitude
noncharacteristic boundary layers has been shown in multi-dimensions
for strictly parabolic systems  in \cite{GG}, and in one dimension
for partially parabolic (``real viscosity'') systems in \cite{R1}.
Here, in the spirit of results obtained for shock waves in
\cite{ZH,MZ3,MZ4},
 we show for general strictly parabolic systems of conservation laws
that linearized and nonlinear stability are equivalent to a
generalized spectral stability condition phrased in terms of the
Evans function associated with the linearized equations about the
wave, {\it independent of the amplitude of the boundary layer in
question.}

The Evans condition is readily checkable numerically, and in some
cases analytically; see \cite{Br1,Br2,BrZ,BDG,HuZ,BHRZ,HLZ}. In
particular, stability of small-amplitude uniformly noncharacteristic
boundary layers has been shown for general hyperbolic--parabolic
systems in multi-dimensions in \cite{GMWZ1} using elementary Evans
function arguments (convergence to the constant layer). An
exhaustive numerical study for isentropic gas layers in one
dimension has been carried out in \cite{CHNZ}, with the conclusion
of stability for 
arbitrary amplitudes.

Our method of analysis is by
pointwise Green function methods like those used in
\cite{ZH,MZ3,MZ4}, and especially \cite{HZ},
to analyze the stability of viscous shock layers.
Similar results have been obtained for the related
small-viscosity-limit problem in \cite{GR,MeZ,GMWZ2}. In particular, we
point to the analysis of Grenier and Rousset \cite{GR} as using
pointwise Green function estimates very similar to those that we use
here, though adapted for different purposes.

\subsection{Equations and assumptions}
Consider a {\it viscous boundary layer}, a standing-wave solution
\eqref{profile} of
a general parabolic system of conservation laws \eqref{parabolic}.
Assume, similarly as in the treatment of the viscous shock case
in \cite{HZ}:
\\

({H0}) \quad  $f,B\in C^{3}$.
\medskip

({H1}) \quad $Re \, \sigma(B) > 0$.
\medskip

({H2}) \quad $\sigma (f'(u_+))$ real, distinct, and  nonzero.
\medskip

({H3}) \quad $Re\,  \sigma(-ik f'(u_+) -k^2 B(u_+))< -\theta k^2$
for all real $k$, some $\theta>0$.
\medskip

({H4}) Solution $\bar{u}$ is unique.
\medskip

\noindent
Here, (H2)(iii) corresponds to noncharacteristicity.
\bigskip
\subsection{Linearized stability and the Evans function}
 After linearizing \ref{parabolic} about the stationary solution
$\bar{u}$, we obtain the linearized equation
\begin{equation}\label{linearized}
u_t=Lu:=-(Au)_x+(Bu_x)_x, \quad A,B \in C^2,
\end{equation}
where
\begin{equation}\label{B}
 B:= B(\bar{u})
\end{equation}
and
\begin{equation}\label{A}
 Au:= dF(\bar{u})u - dB(\bar{u}) (u,\bar{u}_{x}).
\end{equation}

\begin{defi}
The boundary layer  $\bar u$ is said to be {\it linearly
asymptotically stable}, if $u(\cdot,t)$ approaches $0$ as $t\to
\infty$, for any solution $u$ of (\ref{linearized}) with initial
data bounded in in some specified norm.
\end{defi}

We define the following {\it stability criterion}, where
$D(\lambda)$  described below, denotes the Evans function associated
with the linearized operator $L$ about the layer, an analytic
function analogous to the characteristic polynomial of a
finite-dimensional operator, whose zeroes away from the essential
spectrum agree in location and multiplicity with the eigenvalues of
$L$:
\begin{equation}\label{D}
\hbox{\rm
There exist no zeroes of $D(\cdot)$ in the nonstable
half-plane $ Re \lambda \ge 0$}.
\end{equation}
As discussed, e.g., in \cite{R2}, under assumptions (H0)--(H4),
this is equivalent to {\it strong spectral stability},
$\sigma(L)\subset \{Re \lambda < 0\}$, (ii) {\it transversality} of
$\bar u$ as a solution of the connection problem in the associated
standing-wave ODE, and {\it hyperbolic stability} of an associated
boundary value problem obtained by formal matched asymptotics.
Here and elsewhere $\sigma$ denotes spectrum of a
linearized operator or matrix.

Our first main result is as follows.

\begin{theo}\label{linstab}
Assuming (H0)--(H4), linearized asymptotic $L^1\cap L^p\to L^p$ stability,
$p>1$, is equivalent to \eqref{D}.
\end{theo}

Theorem \ref{linstab} is obtained as a consequence of
the following detailed, pointwise bounds on the Green function
$G(x,t;y)$ of the linearized evolution equations
\eqref{linearized} with homogeneous boundary conditions
(more properly speaking, a distribution), defined by:

(i) $(\partial_t -L_x)G=0$ in the distributional sense, for all $x,y,t>0$;

(ii) $G(x, t;y)\rightarrow \delta(x-y)$ as $t\to 0$;

(iii) $G(0, t;y)\equiv 0$, for all $y,t>0$.

\noindent Denote by
\begin{equation}
a_1^+<a_2^+ < \cdots < a_n^+
\end{equation}
the eigenvalues of of the limiting convection matrix $A_+:=
df(u_+)$. \\

Then, our second main result is as follows.

\begin{theo}\label{greenbounds}
Assuming (H0)--(H4) and stability condition \eqref{D},
\begin{equation}\label{Gbounds}
\begin{aligned}
|\partial_{x}^\gamma \partial_y^\alpha & G(x,t;y)|\le  Ce^{-\eta(|x-y|+t)}\\
& +\quad C(t^{-|\alpha|/2}+
|\alpha| e^{-\eta|y|} +|\gamma| e^{-\eta|x|})
\Big( \sum_{k=1}^n
t^{-1/2}e^{-(x-y-a_k^{-} t)^2/Mt} \\
&+
\sum_{a_k^{+} < 0, \, a_j^{+} > 0}
\chi_{\{ |a_k^{+} t|\ge |y| \}}
t^{-1/2} e^{-(x-a_j^{+}(t-|y/a_k^{+}|))^2/Mt} \Big),\\
\end{aligned}
\end{equation}
$0\le |\alpha|, |\gamma| \le 1$, for some $\eta$, $C$, $M>0$, where $x^+$
denotes the positive/negative part of $x$,  indicator function
$\chi_{\{ |a_k^{-}t|\ge |y| \}}$ is $1$ for $|a_k^{-}t|\ge |y|$ and
$0$ otherwise.
\end{theo}

\subsection{Nonlinear stability}\label{nonlinintro}
\begin{defi}
The boundary layer  $\bar u$ is said to be {\it nonlinearly
asymptotically stable} if $\tilde u(\cdot,t)$ exists for all $t\ge0$
and approaches $\bar u$ as $t\to \infty$, for any solution $\tilde u$ of
(\ref{parabolic}) with initial data sufficiently close in some norm
to the original layer $\bar u$.
\end{defi}
\medskip
Denoting by
\begin{equation}
a_1^+<a_2^+ < \cdots < a_n^+
\end{equation}
the eigenvalues of of the limiting convection matrix $A_+:=
df(u_+)$,
\medskip
define
\begin{equation}\label{theta}
\theta(x,t):= \sum_{a_j^+>0}(1+t)^{-1/2}e^{-|x-a_j^+t|^2/Lt},
\end{equation}
\begin{equation}\label{psi1}
\begin{aligned}
\psi_1(x,t)&:= \chi(x,t)\sum_{a_j^+>0}
(1+|x|+t)^{-1/2} (1+|x-a_j^+t|)^{-1/2},\\
\end{aligned}
\end{equation}
and
\begin{equation}\label{psi2}
\begin{aligned}
\psi_2(x,t)&:=
(1-\chi(x,t))(1+|x-a_n^+t|+t^{1/2})^{-3/2},
\end{aligned}
\end{equation}
where $\chi(x,t)=1$ for $x\in [0,a_n^+t]$ and $\chi(x,t)=0$ otherwise
and $L>0$ is a sufficiently large constant.
For simplicity, take $B$ identically constant.
Then, our third and final main result is as follows.

\begin{theo}\label{nonlin}
Assuming (H0)--(H4), $B\equiv \const$,
and the linear stability condition \eqref{D},
the profile $\bar u$ is nonlinearly asymptotically stable with
respect to perturbations $g$, $h$ in initial and boundary data
satisfying
$$
|g(x)|\le E_0 (1+|x|)^{-3/2},
\quad
|h(t)|\le E_0 (1+|t|)^{-3/2},
\quad
|h'(t)|\le E_0 (1+|t|)^{-1}
$$
for $E_0$ sufficiently small. More precisely,
\begin{equation}\label{pointwise}
\begin{aligned}
|\tilde u(x,t)-\bar u(x)|&\le C E_0
(\theta+\psi_1+\psi_2)(x,t),\\
\end{aligned}
\end{equation}
where $\tilde u$ denotes the solution of \eqref{parabolic} with initial data
$\tilde g=\bar u+g$ and boundary data $\tilde h=u_0+h$.
\end{theo}

\begin{rem}
Pointwise bound \eqref{pointwise}
yields as a corollary the sharp $L^p$ decay rate
\begin{equation}
\label{Lp} |\tilde u(x,t)-\bar u(x)|_{L^p}\le C E_0
(1+t)^{-\frac{1}{2}(1-\frac{1}{p})}, \quad 1\le p\le \infty.
\end{equation}
\end{rem}

\subsection{Discussion and open problems}

The case of boundary layers is quite analogous to the
undercompressive shock case; in particular, pointwise
estimates as in \cite{HZ} appear to be necessary to close
the one-dimensional analysis by a linearized semigroup
approach suitable for large-amplitude layers.
(On the other hand, small-amplitude stability has been established
using energy estimates in, e.g., \cite{MN,GG}.)
A new feature of the present analysis as compared to those
of \cite{HZ,HR,HRZ} is the admission
of perturbations in boundary as well as initial data.
Open problems are extensions
to systems with physical (partial) or quasilinear viscosity
and to multi-dimensional boundary layers.

\section{The Evans Function}
Before starting the analysis, we review the basic Evans function
methods and gap/conjugation lemma.

\subsection{The gap/conjugation lemma}
Consider a family of first order ODE systems on the half-line:
\begin{equation}\label{gen}
\begin{aligned}
W'&=\mathbb{A}(x,\lambda)W,
\quad \lambda \in \Omega \quad {\rm and} \quad x>0,\\
\mathbb{B}(\lambda)W&=0, \quad \lambda \in \Omega \quad {\rm and} \quad x=0.
\end{aligned}
\end{equation}
These systems of ODEs should be considered as a generalized
eigenvalue equation, with $\lambda$ representing frequency.
We assume that the boundary matrix $\mathbb{B}$ is analytic
in $\lambda$ and that the coefficient matrix $\mathbb{A}$ is analytic in
$\lambda$ as a function from $\Omega$ into $L^\infty(x)$,
$C^K$ in $x$, and
approaches exponentially to a limit $\mathbb{A}_+(\lambda) $ as $x\rightarrow
\infty$, with uniform exponentially decay estimates

\begin{equation}\label{h0}
|(\partial/\partial x)^k(\BbbA- \BbbA_+)| \le C_1e^{-\theta|x|/C_2},
\, \quad \text{\rm for } x>0, \, 0\le k\le K,
\end{equation}
$C_j$, $\theta>0$, on compact subsets of $\Omega $. Now we can state
a refinement of the ``Gap Lemma'' of \cite{GZ,KS}, relating solutions
of the variable-coefficient ODE to the solutions of its
constant-coefficient limiting equations
\begin{equation}\label{limgen}
 Z'=\mathbb{A}_+(\lambda)Z
\end{equation}
as $x\rightarrow  +\infty $.
\begin{lem}[Conjugation Lemma \cite{MeZ}]
Under assumption \eqref{h0}, there exists locally to any given
$\lambda_0\in \Omega $ a linear transformation
$P_+(x,\lambda)=I+\Theta_+(x,\lambda)$  on $x\ge 0$, $\Phi_+$
analytic in $\lambda$ as functions from $\Omega$ to $L^\infty
[0,+\infty)$, such that:
\medskip

(i) $|P_+|$ and their inverses are uniformly bounded, with
\begin{equation} \label{Theta}
|(\partial/\partial \lambda)^j(\partial/\partial x)^k \Theta_+ |\le
C(j) C_1 C_2 e^{-\theta |x|/C_2} \quad \text{\rm for } x>0, \, 0\le
k\le K+1,
\end{equation}
$j\ge 0$, where $0<\theta<1$ is an arbitrary fixed parameter, and
$C>0$ and the size of the neighborhood of definition depend only on
$\theta$, $j$, the modulus of the entries of $\BbbA$ at $\lambda_0$,
and the modulus of continuity of $\BbbA$ on some neighborhood of
$\lambda_0 \in \Omega $.
\smallskip
\\
(ii)  The change of coordinates $W:=P_+ Z$ reduces \eqref{gen}
on $x\ge 0$ to the asymptotic constant-coefficient equations \eqref{limgen}.
Equivalently, solutions of \eqref{gen} may be conveniently factorized as
\begin{equation}
W=(I+ \Theta_+)Z_+,
\end{equation}
where $Z_+$ are solutions of the constant-coefficient equations, and
$\Theta_+$ satisfy bounds.
\end{lem}

\begin{proof}
As described in \cite{MZ3}, for $j=k=0$ this is a straightforward
corollary of the gap lemma as stated in [Z.3], applied to the
``lifted'' matrix-valued ODE
$$
P'= \mathbb{A}_+P- P\mathbb{A}+ (\mathbb{A}-\mathbb{A}_+)P
$$
for the conjugating matrices $P_+$. The
$x$-derivative bounds $0<k\le K+1$ then follow from the ODE and its
first $K$ derivatives. Finally, the $\lambda$-derivative bounds
follow from standard interior estimates for analytic functions.
\end{proof}

\begin{defi}\label{consplit}
Following \cite{AGJ}, we define the {\it domain of
consistent splitting} for the ODE system $W'=\mathbb{A}(x,\lambda)W$
as the (open) set of $\lambda $ such that the limiting matrix
$\BbbA_+$ is hyperbolic (has no center subspace) and the
boundary matrix $\mathbb{B}$ is full rank, with
$\dim S_+=\Rank \mathbb{B}$.
\end{defi}

\begin{lem}\label{bases}
On any simply connected subset of the domain of consistent
splitting, there exist analytic bases $\{v_{1}, \dots, v_{k}\}^+$
and $\{v_{k+1}, \dots, v_{N}\}^+$ for the subspaces $S_+$ and $U_+$
defined in Definition \ref{consplit}.
\end{lem}

\begin{proof}
By spectral separation of $U_+$, $S_+$, the associated
(group) eigenprojections are analytic. The existence of analytic
bases then follows by a standard result of Kato; see \cite{Kat}, pp.
99--102.
\end{proof}

\begin{cor}\label{stablecor}
By the Conjugation Lemma  , on the domain of consistent splitting,
the stable manifold of solutions decaying as $x\to +\infty$ of \eqref{gen} is
\begin{equation} \label{span}
\CalS^+:=\Span\{P_+v_1^+,\dots,  P_+ v_k^+\},
\end{equation}
where $W_+^j:=P_+v_j^+$ are analytic in $\lambda$ and $C^{K+1}$ in $x$
for $\mathbb{A}\in C^K$.
\end{cor}
\subsection{Definition of the Evans Function}
On any simply connected subset of the domain of consistent
splitting, let $W_1^+, \dots, W_k^+=P_+v_1^+, \dots, P_+v_k^+$
be the
analytic basis described in Corollary \ref{stablecor}
of the subspace $\CalS^+$ of solutions $W$ of \eqref{gen} satisfying
the boundary condition $W\to 0$ at $+\infty$.
 Then, the {\it Evans function} for the ODE
systems $W'=\mathbb{A}(x,\lambda)W$ associated with this choice of
limiting bases is defined as the $k\times k$ Gramian determinant
\begin{equation}\label{deq1}
\begin{aligned}
D(\lambda)&:=
\det \Big( \mathbb{B} W_{1}^+, \dots, \mathbb{B}W_k^+
\Big)_{|x=0, \lambda}\\
&=\det \Big( \mathbb{B}P_+ v_{1}^+, \dots, \mathbb{B}P_+ v_k^+
\Big)_{|x=0, \lambda}.\\
\end{aligned}
\end{equation}
\medskip

\begin{rem}
Note that $D$ is independent of the choice of $P_{+}$ as,
by uniqueness of stable manifolds, the exterior
products (minors) $P_+ v_{1}^+\wedge \dots\wedge P_+ v_k^+$
are uniquely determined by their behavior as $x\to + \infty$.
\end{rem}

\begin{prop}\label{2.4}
Both the Evans function and the subspace $\CalS^+$
are analytic on the entire simply connected subset of the domain of
consistent splitting on which they are defined. Moreover, for
$\lambda$ within this region, equation \eqref{gen} admits a nontrivial
solution $W\in L^2(x>0)$ if and only if $D(\lambda)=0$.
\end{prop}

\begin{proof}
Analyticity follows by uniqueness, and local analyticity
of $P_{+}$, $v_k^{+}$. Noting that the first $P_+v_j^+$
are a basis for the stable manifold of \eqref{gen} at $x\to +\infty$,
we find that the determinant of $\mathbb{B}P_+v_j^+$
vanishes if and only if
$\mathbb{B}(\lambda)$ has nontrivial kernel on $\CalS_+(\lambda,0)$,
whence the second assertion follows.
\end{proof}

\begin{rem}
In this case that the ODE system  describes an eigenvalue equation
associated with an ordinary differential operator $L$, Proposition
\ref{2.4} implies that eigenvalues of $L$ agree in location with zeroes
of $D$.
(Indeed, they agree also in multiplicity; see \cite{GJ1,GJ2};
Lemma 6.1, \cite{ZH}; or Proposition 6.15 of \cite{MZ3}.)
\end{rem}

When $\ker \mathbb{B}$ has an analytic basis
$v^0_{k+1}, \dots, v^0_{N-k}$, for example, in the commonly
occurring case, as here, that $\mathbb{B}\equiv \const$, we have
the following useful alternative formulation.
This is the version that we will use in our analysis of the Green
function and Resolvent kernel.

\begin{prop}\label{evansprop}
Let $v^0_{k+1}, \dots, v^0_{N-k}$
be an analytic basis of $\ker \mathbb{B}$, normalized
so that $\det \big( \mathbb{B}^*, v^0_{k+1}, \dots v^0_{N}\big) \equiv 1$.
Then, the solutions $W^0_j$ of \eqref{gen} determined by initial data
$W^0_j(\lambda,0)=v^0_j$ are analytic in $\lambda$ and $C^{K+1}$
in $x$, and
\begin{equation}\label{deq}
D(\lambda):=
\det \Big( W_{1}^+, \dots,  W_k^+,
W_{k+1}^0, \dots, W_{N}^0 \Big)_{|x=0, \lambda}.
\end{equation}
\end{prop}
\begin{proof}
Analyticity/smoothness follow by analytic/smooth dependence on
initial data/parameters.
By the chosen normalization, and standard properties of Grammian determinants,
$D(\lambda)=
\det \Big( W_{1}^+, \dots,  W_k^+,
v_{k+1}^0, \dots, v_{N}^0 \Big)_{|x=0, \lambda}$,
yielding \eqref{deq}.
\end{proof}
\section{Construction of the Resolvent kernel}
In this section we construct the explicit form of the resolvent
kernel, which is nothing more than the Green function
$G_\lambda(x,y)$ associated with the elliptic operator $(L-\lambda
I)$, where
\begin{equation}\label{Gdef}
(L-\lambda I)G_\lambda(.,y)=\delta_yI,
\quad
G_\lambda(0,y)\equiv 0.
\end{equation}
Let $\Lambda$ be the region of consistent splitting for $L$.
It is an established fact (see [He]) that the resolvent $(L-\lambda I)^{-1}$ and
the Green function $G_\lambda(x,y)$ are meromorphic in $\lambda$
on $\Lambda$, with isolated poles of finite order.
$G_\lambda$ in fact admits a meromorphic extension to a sector
\begin{equation}\label{sector}
\Omega_\theta = \{\lambda : Re (\lambda) \geq -\theta_1 - \theta_2 |Im
(\lambda)|\}, \quad \theta_1,\theta_2 >0.
\end{equation}

Writing the associated eigenvalue equation in
the form of a first-order system \eqref{gen},
we obtain
\begin{equation}
W'={\mathbb{A}}(\lambda, x)W,
\quad \mathbb{B}W(0)=0,
\end{equation}
where
$$
W=
\begin{pmatrix} w\\ w' \end{pmatrix} \in \CC^{2n},
\qquad
\mathbb{A}=
\begin{pmatrix}
0 && I\\
\lambda B^{-1}+A'B^{-1} &&  AB^{-1}-B'B^{-1}
\end{pmatrix}
$$
and $\mathbb{B}\equiv \const$ is the rank-$n$ projection onto the first
coordinate $w$ of $W$, with kernel spanned by the constant basis
$v^0_{n+j}= e_{n+j}$, $j=1, \dots, n$ and $e_j$ the $j$th standard
basis element.

Denote by
\begin{equation}
\Phi^0=
\begin{pmatrix}
\phi^0_1(x;\lambda) & \cdots &
          \phi^0_{n}(x;\lambda)
\end{pmatrix}=
\begin{pmatrix}
W^0_1 & \cdots & W^0_{n}
\end{pmatrix}
\end{equation}
and
\begin{equation} \label{phi+}
\Phi^+=
\begin{pmatrix}
\phi^+_1(x;\lambda) & \cdots &
          \phi^+_{n}(x;\lambda)
\end{pmatrix}=
\begin{pmatrix}
W^+_{n+1} & \cdots & W^+_{2n}
\end{pmatrix}=\begin{pmatrix} P_+v_1^+ & \cdots & P_+v_k^+
\end{pmatrix}
\end{equation}
the matrices whose columns span the subspaces of
solutions of \eqref{gen} decaying at $x=0,+\infty$
respectively, denoting (analytically chosen) complementary subspaces by
\begin{equation}
\Psi^0=
\begin{pmatrix}
\psi^0_1(x;\lambda) & \cdots &
          \psi^0_{n}(x;\lambda)
\end{pmatrix}=
\begin{pmatrix}
W^0_{n+1} & \cdots & W^0_{2n}
\end{pmatrix}
\end{equation}
and \\
\begin{equation}
\Psi^+=
\begin{pmatrix}
\psi^+_1(x;\lambda) & \cdots &
          \psi^+_{n}(x;\lambda)
\end{pmatrix}=
\begin{pmatrix}
W^+_{1} & \cdots & W^+_{n}
\end{pmatrix}.
\end{equation}
As described in the previous subsection,
eigenfunctions decaying at both $0,+\infty$ occur precisely when
the subspaces $\Span \Phi^0$ and $\Span \Phi^+$ intersect,
i.e., at zeros of the Evans function defined in \eqref{deq}:
\begin{equation}
D_L(\lambda):=\det(\Phi^0,\Phi^+)_{|x=0}=\begin{pmatrix} \phi^0_1
\wedge \cdots \wedge \phi^0_n \wedge \phi^+_1 \wedge \cdots \wedge
\phi^+_n\end{pmatrix}_{|x=0}.
\end{equation}

\begin{lem}[\cite{GZ,ZH}]
For $\theta_1,\theta_2 >0$ sufficiently small, $D_L$ is locally
analytic on sector $\Omega_\theta$ as defined in \eqref{sector}.
\end{lem}

\begin{proof} Direct calculation showing that the domain $\Lambda$
of consistent splitting is contained in $\Omega_\theta-B(0,r)$
for $r>0$ arbitrary and $\theta$ sufficiently small, with $v_j^+$
extending analytically to $B(0,r)$.
\end{proof}

\begin{lem}\label{greendual}
Let $H_\lambda(x,y)$ denote the Green function for the adjoint
operator $(L-\lambda I)^*$ on the half-plane $x\ge 0$.
Then $G_\lambda(x,y)=H^*_\lambda(x,y)$.
In particular, for $x\neq y$, the matrix $z=G_\lambda(x,.)$
satisfies
\begin{equation}\label{Zeq}
(z'B)'=-z'A+z\lambda.
\end{equation}
\end{lem}

\begin{proof}
Standard duality argument; see \cite{ZH} for operators on the whole line.
\end{proof}

Considering \eqref{Zeq} as an ODE system for the vector $Z=(z,z')$, it
becomes
\begin{equation}
Z'=Z\tilde{\mathbb{A}}(\lambda, x),
\end{equation}
where
\begin{equation}\label{firstorderZ}
\tilde{\mathbb{A}}=
\begin{pmatrix}
0 && \lambda B^{-1}-A'B^{-1} \\
I && -AB^{-1}-B'B^{-1}
\end{pmatrix}.
\end{equation}

\begin{lem}[\cite{ZH}]\label{duality}
$Z$ is a solution of \eqref{firstorderZ} if and only if
$Z \CalS W \equiv$ constant
for any solution W of \eqref{gen},
where $\CalS = \begin{pmatrix} -A & B \\
-B & 0 \end{pmatrix}.$
\end{lem}

\begin{proof}
Direct computation/comparison with $0$ of $(Z \CalS W)'$;
see \cite{ZH}.
\end{proof}

Using Lemma \ref{duality}, we can define dual bases $\tilde{W}^0_j$ and
$\tilde{W}^+_j$ by the relations
\begin{equation}
\tilde{W}^{0,+}_j\CalS W^{0,+}_k=\delta^j_k.
\end{equation}

Likewise, $\tilde{\mathbb{A}}_{0,+}$ can be defined as
\begin{equation}
\tilde{\mathbb{A}}_{0,+}=\begin{pmatrix} 0 & \lambda B_{0,+}^{-1}\\
I & -A_{0,+}B^{-1}_{0,+} \end{pmatrix}.
\end{equation}
We define also the dual subspaces
\begin{equation}
\tilde{\Phi}^0=
\begin{pmatrix}
\tilde{\phi}^0_1(x;\lambda) & \cdots &
          \tilde{\phi}^0_{n}(x;\lambda)
\end{pmatrix}=\begin{pmatrix}
\tilde{W}^0_{n+1} & \cdots & \tilde{W}^0_{2n}
\end{pmatrix},
\end{equation}
\begin{equation}\label{tildephi+}
\tilde{\Phi}^+=
\begin{pmatrix}
\tilde{\phi}^+_1(x;\lambda) & \cdots &
          \tilde{\phi}^+_{n}(x;\lambda)
\end{pmatrix}=\begin{pmatrix}
\tilde{W}^+_{1} & \cdots & \tilde{W}^+_{n}
\end{pmatrix},
\end{equation}

\begin{equation}
\tilde{\Psi}^0=
\begin{pmatrix}
\tilde{\psi}^0_1(x;\lambda) & \cdots &
          \tilde{\psi}^+_{n}(x;\lambda)
\end{pmatrix}=\begin{pmatrix}
\tilde W_1^0& \cdots&  \tilde W_n^0
\end{pmatrix},
\end{equation}

\begin{equation}\label{tildepsi+}
\tilde{\Psi}^+=
\begin{pmatrix}
\tilde{\psi}^+_1(x;\lambda) & \cdots &
          \tilde{\psi}^0_{n}(x;\lambda)
\end{pmatrix}=\begin{pmatrix}
 \tilde W_{n+1}^+& \cdots& \tilde W_{2n}^+
\end{pmatrix}.
\end{equation}

With these preparations, the construction of the Resolvent
kernel goes exactly as in the construction performed
in \cite{ZH,MZ3} on the whole line.

\begin{lem}\label{rep}
We have the the representation
\begin{equation}
\begin{pmatrix} G_\lambda & G_{{\lambda}_y} \\
                G_{{\lambda}_x} & G_{{\lambda}_{xy}}
\end{pmatrix}=
\begin{cases}
   \Phi^+(\lambda,x)M^+(\lambda)\tilde{\Psi}^0(\lambda,y)
                                  \quad & for \quad x>y,\\
   \Phi^0(\lambda,x)M^0(\lambda)\tilde{\Psi}^+(\lambda,y)
                                  \quad &for \quad x<y,\\
\end{cases}
\end{equation}
where $M^{0,+}$ are to be determined.
\end{lem}

\begin{proof}
See \cite{ZH} Lemma 4.6.
\end{proof}

Using Lemma \ref{rep}, we find the explicit coordinate-free
representation for $x>y$:
\begin{equation}
\begin{pmatrix} G_\lambda & G_{{\lambda}_y} \\
                G_{{\lambda}_x} & G_{{\lambda}_{xy}}
\end{pmatrix}=\CalF^{z\rightarrow x}
\Pi_+(z)\CalS^{-1}(z)\tilde{\Pi}_0(z)\tilde{\CalF}^{z\rightarrow y},
\end{equation}
where \\
\begin{equation}
\Pi_+(y)=(\Phi^+(y),0)(\Phi^+(y),\Phi^-(y))^{-1},
\end{equation}

\begin{equation}
\tilde{\Pi}_0(y)=\begin{pmatrix}\tilde{\Psi}^0(y)\\
\tilde{\Psi}^+(y)\end{pmatrix}^{-1}
\begin{pmatrix}\tilde{\Psi}^0(y)\\
0\end{pmatrix},
\end{equation}
\begin{equation}
\CalF^{z\rightarrow
x}=(\Phi^+(x),\Phi^0(x))(\Phi^+(z),\Phi^0(z))^{-1},
\end{equation}
\begin{equation}
\tilde{\CalF}^{z\rightarrow y}=\begin{pmatrix}\tilde{\Psi}^0(z)\\
\Phi^+(z)\end{pmatrix}\begin{pmatrix}\Psi^0(y)\\
\Psi^+(y)\end{pmatrix}^{-1},
\end{equation}
and similarly for $x<y$.
\begin{cor}
 The resolvent kernel may  be expressed as
 \begin{equation}
  G_\lambda(x,y)=
  \begin{cases}
   (I_n,0)\Phi^+(x;\lambda)M^+(\lambda)\tilde\Psi^{0*}(y;\lambda)(I_n,0)^{tr}\ &x>y,\\
   -(I_n,0)\Phi^0(x;\lambda)M^0(\lambda)\tilde\Psi^{+*}(y;\lambda)(I_n,0 )^{tr}\ &x<y,
                 \end{cases}
 \end{equation}
where
\begin{equation}
  M(\lambda):=\text{\rm diag}(M^+(\lambda),M^0(\lambda))=
  \Phi^{-1}(z;\lambda)\bar \CalS^{-1}(z)\tilde\Psi^{-1*}(z;\lambda).
\end{equation}
\end{cor}
\section {Low-frequency bounds}
\smallskip
 Our goal in this section is the estimation of the resolvent kernel
 in the critical regime $|\lambda|\to 0$, i.e., the large time
 behavior of the Green function G, or global behavior in space and
 time. We are basically following the same treatment as that
carried out for viscous shock waves of strictly parabolic conservation laws
in \cite{ZH,MZ3}; we refer to those references for details.
 In the low frequency case the behavior is essentially governed by
 the equation
\begin{equation}
U_t=L_+U :=-A_+U_x+B_+U_{xx}
\end{equation}
\begin{prop}
Assuming (H0)-(H4), let $K$ be the order of the pole of $G_\lambda$ at
$\lambda=0$ and r be sufficiently small that there are no other
poles in $B(0,r)$.
Then for $\lambda \in \Omega_\theta$ such that
$|\lambda|\leq r$ and for $x>y>0$ we have
\begin{equation} \label{greenlow}
\begin{pmatrix} G_\lambda & G_{{\lambda}_y} \\
                G_{{\lambda}_x} & G_{{\lambda}_{xy}}
\end{pmatrix}=\sum_{j,k}d_{jk}(\lambda)\phi_j^+(x)\tilde{\psi}_k^+(y)+\sum_{j,k}(\lambda)\phi_k^+(x)\tilde{\phi}_k^+(y),
\end{equation}
where $d_{jk}(\lambda)=\CalO(\lambda^{-K})$ and
$e_{jk}(\lambda)=\CalO(\lambda^{1-K})$ are scalar meromorphic
functions, moreover $K \leq$ order of vanishing of the Evans
function $D(\lambda)$ at $\lambda=0$.
\end{prop}

\begin{proof}
See \cite{ZH} Proposition 7.1 for the first statement and theorem
6.3 for the second statement linking order $K$ of the pole to
multiplicity of the zero of the Evans Function.
\end{proof}

\begin{lem}
Assuming (H0)-(H4), for $|\lambda|$ sufficiently small, the
eigenvalue equation $(L_{+}-\lambda)W = 0$ associated with the
limiting, constant-coefficient operator $L_+$ has a basis of $2n$
solutions $ \bar{W}^{+}_j=e^{\mu^{+}_j(\lambda)x}V_{j}(\lambda)$
where $\mu^{+}_{j}$ and $V^{+}_{j}$are analytic in $\lambda$,
consisting of n fast modes
 \begin{equation}
 \begin{aligned}
\mu_{j}^{+}&=\gamma^{+}_{j}+\CalO (\lambda),\\
  V_{j}^{+}&=  S^{+}_{j} + \CalO(\lambda),\\
 \end{aligned}
  \end{equation}
where $\gamma^+_j$, $S_j^+$ are
eigenvalues and associated right eigenvectors
of $B_+^{-1}A_+$,
and $n$ slow modes
\begin{equation} \label{mubound}
\begin{aligned}
\mu_{r+j}^{+}(\lambda)&:=
-\lambda/a^{+}_{j}+\lambda^2\beta^{+}_{j}/a^{+^{3}}_{j}+\CalO(\lambda^3),\\
V_{r+j}^{+}(\lambda) &:= r^{+}_{j} +\CalO(\lambda),\\
\end{aligned}
\end{equation}
where $a_j^{+}$, $l_j^+$, $r_j^{+}$ are eigenvalues and left and
right eigenvectors of $A_+:=dF(u_+)$ , and $\beta_j^+:=l_j^+B_+r_j^+>0$
with $B_+:=B(u_+)$.
The same is true for the adjoint eigenvalue equation
$$
(L^+-\lambda)^*Z=0,
$$
i.e, it has a basis of solutions
$$ \bar{ \tilde
W}_{j}^+ = e^{-\mu_j^+(\lambda)x}\tilde{V}_j(\lambda)
$$
with
\begin{equation}
\tilde{V}_{j}^+(\lambda)= \tilde T_j^{+} + \CalO(\lambda),
\end{equation}
\begin{equation}
\tilde{V}_{r+j}^+(\lambda)= l_j^+ + \CalO(\lambda),
\end{equation}
$\tilde V^+$ analytic in $\lambda$.
\end{lem}

\begin{proof}
See \cite{MZ3}.
\end{proof}

\begin{prop}
Assume (H0)-(H4) and \eqref{D}, then, for $r>0$ sufficiently small,
the Resolvent kernel $G_\lambda$ associated with the  linearized
evolution equation
\begin{equation}\label{lineq}
U_t=L_+U:=-A_+U_x+B_+U_{xx}
\end{equation}
satisfies, for $0\le  y \le x$:
\begin{equation}\label{ptres}
\begin{aligned}
|\partial_x^\gamma \partial_y^\alpha G_\lambda(x,t;y)|&\le C
(|\lambda|^{\gamma} + e^{-\theta |x|}) (|\lambda|^{\alpha} + e^{-\theta
|y|}) \Big(\sum_{a_k^{+}>0} \big|e^{ (-\lambda/a^{+}_k +
\lambda^2 \beta^{+}_k/{a^{+}_k}^3 )(x-y)}\big|\\
&\quad +\sum_{a_k^{+} < 0, \,  a_j^{+} > 0}
\big|e^{(-\lambda/a^{+}_j + \lambda^2 \beta^{+}_j/{a^{+}_j}^3
)x +(\lambda/a^{+}_k - \lambda^2 \beta^{+}_k/{a^{+}_k}^3 )y}\big|
\Big),
\end{aligned}
\end{equation}
$0\le |\alpha|, |\gamma| \le 1$, $\theta>0$,
with similar bounds for $0\le x\le y$.
Moreover, each term in the summation on the righthand side of \eqref{ptres}
bounds a separately analytic function.
\end {prop}

\begin{proof}
By 1.8 $D$ does not vanish on $Re(\lambda)\geq 0$, hence, by continuity,
on $|\lambda|\le r$.
Thus, according to
\eqref{greenlow}, all $|d_{jk}(\lambda)|$ are uniformly bounded
on $|\lambda|\le r$,
and so it is enough to find estimates for fast and slow
modes $\phi_j^+$, $\tilde{\phi}^+_j$, $\psi_j^+$ and
$\tilde{\psi}^+_j$. By using \eqref{phi+} we find:
\begin{equation}  \label{lphi+}
 \begin{pmatrix} \phi^+_j \\ \partial_x \phi^+_j
 \end{pmatrix}
=e^{\mu_j(\lambda)x}P^+\begin{pmatrix}V_j \\ \mu_j
 V_j\end{pmatrix}=e^{\mu_j(\lambda)x}(I+\Theta)\begin{pmatrix}V_j \\ \mu_j
 V_j\end{pmatrix}
\end{equation}
and similarly for $\tilde{\phi}^+_j$, $\psi_j^+$ and
$\tilde{\psi}^+_j$. Now using
\eqref{Theta} and
the fact, by \eqref{mubound}, that
$e^{\mu_j(\lambda)x}$ is of order $e^{-| \theta x |}$ for fast modes
and order
$e^{ -\lambda/a^{+}_{j}+\lambda^2\beta^{+}_{j}/a^{+^{3}}_{j}+\CalO(\lambda^3)}$
for slow modes,
substituting this and corresponding dual estimates
in \eqref{lphi+} and grouping terms, we obtain
the result.
%
\end{proof}


\section{High frequency bounds}
To analyze the high frequency behavior of the Green function of
the boundary layer, we first establish some bounds for the projection
terms in the Green function, using the symmetric formula

\begin{equation}\label{repeq}
\begin{pmatrix}
G_\lambda(x,y) & \partial_y G_\lambda(x,y) \\
\partial_x G_\lambda(x,y) & \partial_x \partial_y G_\lambda(x,y) \\
\end{pmatrix}=
\begin{cases}
   \CalF^{y\rightarrow x} \Pi_+(x)\CalS^{-1}(y) \quad &if \quad x>y,\\
   \tilde{\CalF}^{x\rightarrow y} \tilde{\Pi}_+(x)\CalS^{-1}(x)
\quad& if \quad x<y.
\end{cases}
\end{equation}

By setting $\bar{x}=|\lambda ^\frac{1}{2}|x$,
 $\bar{\lambda}=\frac{\lambda}{|\lambda|}$, $\bar{B}(\bar{x})=
 B(\frac{\bar{x}}{\lambda^\frac{1}{2}})$, $\bar{w}(\bar{x})=w(\frac{x}{\lambda^\frac{1}{2}})$
in the eigenvalue equation $Lw=\lambda w$ associated with
  \eqref{linearized} we obtain
  \begin{equation}
  \bar{W}'=\bar{\mathbb{B}}\bar{W}+\CalO(|\lambda^{-\frac{1}{2}}|)\bar W
  \end{equation}
  where
  \begin{equation}
  \mathbb{B}=\begin{pmatrix}
  0&I \\ \bar {\lambda}\bar{B}&0
  \end{pmatrix}
  \end{equation}
  and $\mathbb{B}'=\CalO {(|\lambda^{-\frac{1}{2}}|)}$ and
  $|\bar \lambda|=1$. Since $\mathbb{B}(\lambda,\bar{x})$ varies within a compact
  set, then
  there are $C^1$ eigenprojections
  $P_0$ and $P_+$ with property $|P'_+|=\CalO(|\lambda^{-\frac{1}{2}}|)$
  and $|P'_0|=\CalO(|\lambda^{-\frac{1}{2}}|)$ taking $\bar
  W$ onto the stable and unstable subspace.
  By using the two new coordinates $Y_+=P_+\bar W$ and $Y_0=P_0\bar
  W$, we obtain
  \begin{equation}
  \begin{pmatrix} Y_+ \\ Y_0 \end{pmatrix}'=\begin{pmatrix}
  A_+&0\\0&A_0\end{pmatrix}\begin{pmatrix} Y_+ \\ Y_0
  \end{pmatrix}+\CalO(|\lambda^{-\frac{1}{2}}|)\begin{pmatrix}y \\ y \end{pmatrix}.
  \end{equation}
Equivalently, we can find continuous invertible transformations
$Q_+$, $Q_0$ such that $E_+=Q_+A_+Q_+^{-1}$ and $E_0=Q_0A_0Q_0^{-1}$
where
\begin{equation} \label{beta}
 Re(E):=\frac{1}{2}(E_++E^*_+)< -\beta^{-\frac{1}{2}}I.
\end{equation}
in the sense of quadratic forms.\\
 Again by coordinate change $Z_+=Q_+Y_+$, $Z_0=Q_0Y_0$ we find
\begin{equation}
  \begin{pmatrix} z_+ \\ z_0 \end{pmatrix}'=\begin{pmatrix}
  E_+&0\\0&E_0\end{pmatrix}\begin{pmatrix} z_+ \\ z_0
  \end{pmatrix}+\CalO(|\lambda^{-\frac{1}{2}}|)\begin{pmatrix}z \\ z \end{pmatrix}
  \end{equation}
where
\begin{equation} \label{C}
\frac{|\bar w|}{C} \leq |z| \leq C|\bar w|.
\end{equation}
From this we find by energy estimate that
\begin{equation}
(|z_+|^2)'< -2\beta^{-\frac{1}{2}}|z_+|^2
\end{equation}
and hence
\begin{equation}
\frac{|z_+(x)|}{|z_+(y)|} \leq
e^{-\tilde{\beta}^{-\frac{1}{2}}|(x-y)|}
\end{equation}
for any solution$z_+$ decaying at $\infty$, where
$\tilde{\beta}<\beta$ and thus
\begin{equation}
\frac{|z_(x)|}{|z_(y)|} \leq e^{-\beta^{-\frac{1}{2}}|(x-y)|}
\end{equation}
for $x>y$, provided that $|\lambda|$ is sufficiently large. From
this we obtain
\begin{equation} \label{Wbound}
\frac{|\bar W(x)|}{|\bar W(y)|} \leq C^2
e^{-\beta^{-\frac{1}{2}}|(x-y)|}
\end{equation}
where $C$ is as in \eqref{C}.
Applying a symmetric argument for the adjoint equation, we
obtain the following lemma.


\begin{lem}\label{flowlemma}
On the manifolds $\Phi_+$ and $\tilde{\Psi}_+$
defined in \eqref{phi+} and \eqref{tildepsi+},
for $\lambda$ sufficiently large, within the
sector $\Omega_\theta = \{\lambda : Re (\lambda) \geq -\theta_1 -
\theta_2 |Im (\lambda)|\}$, $\theta_1,\theta_2 >0,$
we have in rescaled coordinates $\bar x$, for some uniform $C>0$,
\begin{equation} \label{calF}
|\CalF^{y\rightarrow x}|, \,
|\tilde{\CalF}^{x\rightarrow y}|\leq
Ce^{-\frac{|y-x|}{C}}
\end{equation}
for $x>y$ and $x<y$ respectively.
\end{lem}

\begin{lem}\label{projlemma}
In rescaled coordinates $\bar x$, $\bar \lambda$,
for
the projection terms $\Pi_+(y)$ and $\tilde{\Pi}_+(x)$, the
projection along $\Phi_0$ onto $\Phi_+$, for $\lambda$ sufficiently
large,
\begin{equation} \label{Pi(y)}
|\Pi_+(y)|, \, |\tilde{\Pi}_+(x)|<C
\end{equation}
for some uniform $C>0$.
\end{lem}

\begin{proof}

Choosing the coordinates $\begin{pmatrix} W_1 \\ W_2
\end{pmatrix} \in \mathbb{C}^{2n}$  where $W^j=\begin{pmatrix} W^j_1 \\ W^j_2
\end{pmatrix}$,
we show for small enough $\epsilon$ and fixed $c>0$
such that$\frac{|W_2^j|}{|W_1^j|}\leq c\epsilon$ and
$\frac{|W_1^j|}{|W_2^j|}<c$ the projection along $E=\Span(W^1...W^n)$
onto $F=\Span(W^{n+1}...W^{2n})$\\
\begin{equation}
\Pi:=(w^1...w^{2n})(O_n,I_n)(w^1...w^{2n})^{-1}
\end{equation}
satisfies
\begin{equation}
|\Pi|\leq C_2(c,\epsilon)
\end{equation}
To show this without loss of generality we assume that
\begin{equation}
(w^{n+1}...w^{2n})=
\begin{pmatrix}
I_n \\ \mathcal{O}(\epsilon)
\end{pmatrix}
(w^1...w^n)=
\begin{pmatrix}
 \mathcal{O}(1)\\ I_n
\end{pmatrix}.
\end{equation}
Now it is sufficient to show that
\begin{equation}
|(w^1,...,w^{2n})|\leq C_2(c,\epsilon).
\end{equation}
But, this amounts to showing that
\begin{equation}
\Big|\begin{pmatrix} M & I_n \\ I_n & O \end{pmatrix} +
\mathcal{O}(\epsilon))^{-1}\Big| \leq C,
\end{equation}
which amounts to showing that
\begin{equation}
\Big|\begin{pmatrix} M & I_n \\ I_n & O \end{pmatrix}^{-1}\Big| \leq C_2(c),
\end{equation}
where $|M|\leq c$. But, this is easy to show because
$$
\begin{pmatrix} M & I_n \\ I_n & O
\end{pmatrix}^{-1}=\begin{pmatrix} I_n & -M \\ O & I_n \end{pmatrix},
$$
and so
$$
\Big|\begin{pmatrix} M & I_n \\ I_n & O \end{pmatrix}^{-1}\Big|
\leq 1+ |M| \leq C.
$$
\end{proof}

\begin{prop}
Assume (H0)-(H4) and \eqref{D}. Then, for $R>0$ sufficiently large,
the Resolvent kernel $G_\lambda$ associated with the  linearized
evolution equation \eqref{lineq} satisfies, for
$c,C>0$ and
$0\le |\alpha|, |\gamma| \le 1$:
\begin{equation}\label{hfbdeq}
\begin{aligned}
|\partial_x^\gamma \partial_y^\alpha G_\lambda(x,y)|&\le
C|\lambda|^{(\frac{|\alpha|+|\gamma|-1}{2})}
e^{-\sqrt{\lambda}\frac{|y-x|}{c}}
\end{aligned}
\end{equation}
\end{prop}

\begin{proof}
Recalling the coordinate-free representation \eqref{repeq}
and combining with \eqref{calF} and \eqref{Pi(y)}, we
find that the Green function $\bar G_{\bar \lambda}$
in rescaled coordinates $\bar x$, $\bar \lambda$ satisfies
\begin{equation}
\begin{aligned}
|\partial_x^\gamma \partial_y^\alpha
\bar G_{\bar \lambda}(\bar x,\bar y)|&\le
C e^{-\frac{|y-x|}{C}},
\end{aligned}
\end{equation}
whence \eqref{hfbdeq} follows in the original coordinates.
\end{proof}

\begin{rem}
The argument of Lemma \ref{projlemma} is the key new ingredient in the
resolvent estimates for the boundary layer case as compared to the
analysis on the whole line carried out for viscous shock layers in
\cite{ZH}, making essential use of compatibility of the boundary
condition with high-frequency behavior. On the whole line, there is
no such requirement and high-frequency stability is automatic.
\end{rem}
\section{Pointwise Green function bounds}

With the pointwise bounds established on the resolvent kernel
$G_\lambda$, we obtain pointwise bounds on the Green
function through the inverse Laplace transform formula by
a simplified version of the stationary-phase arguments used
in \cite{ZH} for the shock case, repeated here for completeness.

\begin{proof}[Proof of Theorem \ref{greenbounds}]
By sectoriality of L, we have the inverse Laplace transform
representation (see \cite{ZH}):
\\
\begin{equation}\label{iLT}
G(t;x,y)=\int_{\Gamma}e^{\lambda t}G_{\lambda}(x,y)d\lambda.
\end{equation}
Let $\theta_1 > 0,\ \theta_2 > 0$ be chosen sufficiently small, in
particular so small as to satisfy the hypotheses of all previous
assertions.
By assumption \eqref{D}, the large-$|\lambda|$
bounds on the Resolvent kernel, and analyticity
of the Evans function $D_L(\lambda)$, it follows that
$G_\lambda$ has finitely many poles in $\Omega_\theta$
(corresponding to roots of $D_L$), each with strictly negative real part.
Choosing $\theta_1,\theta_2$ still smaller, if necessary, we can thus arrange
that {\it $G_\lambda$ is analytic on $\Omega_\theta$}.
It follows from Cauchy's Theorem that
\begin{equation}
G(x,t;y)=\int_\Gamma e^{\lambda t} G_\lambda(x,y)\, d\lambda,
\end{equation}
for any contour $\Gamma$ that can be expressed as
$ \Gamma=\partial(\Omega_\theta \setminus \CalS) $ for $\CalS\subset \CC$ open.
\bigskip

{\bf Case I. $|x-y|/t$ large.} We first treat the trivial case that
$|x-y|/t\ge S$, $S$ sufficiently large, the regime in which standard
short-time parabolic theory applies. Set
\begin{equation}\label{largedef}
\bar{\alpha} :={\frac{|x-y|}{2 \beta t}}, \quad R:= \beta
\bar{\alpha}^2,
\end{equation}
where $\beta $ is as in \eqref{beta}, and consider again the
representation of $G$, that is
\begin{equation}
G(x,t;y)=\int_{\Gamma_1\cup \Gamma_2} e^{\lambda t} G_\lambda(x,y)
\, d\lambda,
\end{equation}
where $\Gamma_1:= \partial B(0,R)\cap \bar \Omega_\theta$ and
$\Gamma_2:= \partial \Omega_\theta \setminus B(0,R)$.
Note that the intersection of $\Gamma$ with the real axis is
$\lambda_{min}=R=\beta \bar{\alpha}^2$.
\end{proof}
By the large $|\lambda|$ estimates of Proposition 5.3, we have for
all $\lambda \in \Gamma_1\cup \Gamma_2$ that
$$
|G_\lambda (x,y)|\le C
\frac{e^{-\sqrt{|\lambda|}\frac{|y-x|}{c}}}{\sqrt{|\lambda|}}.
$$
Further, we have
\begin{equation}
\begin{aligned}
Re \lambda &\le  R(1- \eta \omega^2), \quad \lambda\in \Gamma_1, Re
\lambda &\le Re \lambda_0 - \eta (|Im \lambda| - |Im \lambda_0|),
\quad \lambda \in \Gamma_2
\end{aligned}
\end{equation}
for $R$ sufficiently large, where $\omega$ is the argument of
$\lambda$ and $\lambda_0$ and $\lambda_0^*$ are the two points of
intersection of $\Gamma_1$ and $\Gamma_2$, for some $\eta>0$
independent of $\bar{\alpha} $. Combining these estimates, we obtain
\begin{equation}
\begin{aligned}
|\int_{\Gamma_{1}} e^{\lambda t} G_\lambda  d\lambda| &\le
\int_{\Gamma_{1}}C |\lambda|^{-\frac{1}{2}}\, e^{Re \lambda t
-\beta^{-\frac{1}{2}} |\lambda|^{-\frac{1}{2}}|x-y| } \, d\lambda
\cr &\le C e^{-\beta \bar{\alpha} ^{2}t} \int_{-L}^{+L}
R^{-\frac{1}{2}}e^{-\beta R \eta \omega^2 t} \, R d\omega &\le
Ct^{-\frac{1}{2}} e^{-\beta \bar{\alpha} ^{2}t}.
\end{aligned}
\end{equation}
Likewise,
\begin{equation}
\begin{aligned}
|\int_{\Gamma_{2}} e^{\lambda t} G_\lambda  d\lambda| &\le
\int_{\Gamma_{2}}C |\lambda|^{-\frac{1}{2}}\, C e^{Re \lambda t
-\beta^{-\frac{1}{2}} |\lambda|^{-\frac{1}{2}}|x-y|} d\lambda \cr
&\le C
e^{Re(\lambda_{0})t-|\beta|^{-\frac{1}{2}}|\lambda_0|^{-\frac{1}{2}}
|x-y|} \int_{\Gamma_{2}} |\lambda|^{-\frac{1}{2}}e^{(Re
\lambda-Re\lambda_{0})t}\ |d\lambda|\cr
&\le C e^{-\beta \bar{\alpha} ^{2}t} \int_{\Gamma_2}  |Im \,
\lambda|^{-\frac{1}{2}} e^{-\eta|Im \, \lambda-Im \, \lambda_{0}|t}\
|d\, Im \lambda|\cr
& \le  Ct^{-\frac{1}{2}} e^{-\beta \bar{\alpha}^{2}t}. \cr
\end{aligned}
\end{equation}
Combining these last two estimates, we have
\begin{equation} \label{g1}
|G(x,t;y)|\le Ct^{-\frac{1}{2}} e^{\frac{-\beta \bar{\alpha}
^{2}t}{2}} e^{\frac{-(x-y)^{2}}{8\beta t}} \le
Ct^{-\frac{1}{2}}e^{-\eta t} e^{\frac{-(x-y)^{2}}{8\beta t}},
\end{equation}
for $\eta>0$ independent of $\bar{\alpha}$. Observing that
${|x-at|\over 2t} \le {\frac{|x-y|}{t}} \le{\frac{2|x-at|}{t}}$ for
any bounded $a$, for $\frac{|x-y|}{t}$ sufficiently large, we find
that this contribution may be absorbed in any summand $
t^{\frac{-1}{2}}e^{\frac{-(x-y-a_k^+ t)^2}{Mt}}.  $
\medskip \\
{\bf Case II. $|x-y|/t$ bounded.} We now turn to the critical case
that $|x-y|/t \le S$. A few remarks are in order at the outset. Our
goal is to bound $|G|$ by terms of form $C
t^{-1/2}e^{-{\bar{\alpha}}^2 t/M}$, where $\bar{\alpha}:=
(x-a_j^+(t-|y/a_k^+|)/2t$ or $\bar{\alpha}:=(x-y-a_k^+ t)/2t$ are
now {\it uniformly bounded}, by
\begin{equation}
|x-y|/2t + \max_j\{|a_j^+|\}/2 \le S/2 + \max{|a_j^+|}/2.
\end{equation}

Thus, in particular, {\it contributions of order $t^{-1/2}e^{-\eta
t}$, $\eta>0$, can be absorbed in any summand
$ t^{-1/2}e^{-(x-y-a_k^+ t)^2/Mt}) $} if we take $M$ sufficiently large.
Likewise, for $G_x$ and $G_y$,
contributions of order $t^{-1}e^{-\eta t}$ can be absorbed.
We will use this observation repeatedly.

In contrast to the previous case of large characteristic speed
$|x-y|/t\ge S$, we are not trying to show rapid time-exponential
decay.  Rather, we are trying to show that the rate of exponential
decay of the solution does not degrade too rapidly as $\bar{\alpha}
\to 0$: precisely, that it vanishes to order ${\bar{\alpha}}^2$ and
no more. Thus, the crucial part of our analysis will be for small
$\bar{\alpha}$. All other situations can be estimated crudely as
described just above.

Let $r$ be sufficiently small that the small-$|\lambda|$ bounds hold
on $B(0,r)$. Next, choose $\theta_1$ and $\theta_2$ still smaller
than before, if necessary, so that $\Omega_\theta \setminus B(0,r)
\subset \Lambda$. This implies that $\partial \Omega_\theta \cap
B(0,r)\ne \emptyset$, giving the configuration pictured in the
Figure. Similarly as in the previous case, define $\Gamma = \Gamma_1
\cup \Gamma_2$, where $\Gamma_1$ is the portion of the circle
$\partial B(0,r)$ contained in $\bar \Omega_\theta$, and $\Gamma_2$
is the portion of $\partial \Omega_\theta$ outside $B(0,r)$.

\begin{equation}
G(x,t;y)=\int_{\Gamma_1}e^{\lambda t} G_\lambda(x,y) \, d\lambda
+ \int_{\Gamma_2}e^{\lambda t} G_\lambda(x,y) \, d\lambda,
\end{equation}
we separately estimate the terms $\int_{\Gamma_1}$ and
$\int_{\Gamma_2}$.
\\
{\bf Large- and medium-$\lambda$ estimates.} The $\int_{\Gamma_2}$
term is straightforward. The points $\lambda_0$, $\lambda^*_0$ where
$\Gamma_1$ meets $\Gamma_2$ satisfy $Re(\lambda_0) = -\eta <0$.
Moreover, combining the results low frequency case, we have the
bound $|G_\lambda| \le C|\lambda|^{-\frac{1}{2}} $ for $\lambda \in
\Gamma_2$.
Thus, we have
\begin{equation}
\begin{aligned}
|\int_{\Gamma_{2}} e^{\lambda t} G_\lambda  d\lambda| &\le C e^{-Re
\, \lambda_0 t} \int_{\Gamma_2}  |Im \, \lambda|^{-\frac{1}{2}}
e^{-\eta|Im \, \lambda-Im \, \lambda_{0}|t}\ |d\, Im \lambda|\cr &
\le Ct^{-\frac{1}{2}} e^{-\eta t}.
\end{aligned}
\end{equation}
This contribution can be absorbed as described above.
An analogous computation using $|G_{\lambda_x}|,
\, |G_{\lambda_x}| \le C|\lambda|^{-1}$ shows that the $\Gamma_2$
contribution to $G_x$ and $G_y$ is $O(t^{-1} e^{-\eta t})$,
and can likewise be absorbed.\\
\\
{\bf Small $|\lambda|$ estimates}. It remains to estimate the
critical term $\int_{\Gamma_1} e^{\lambda t} G_\lambda d \lambda$.
This we will estimate in different ways, depending on the size of
$t$.\\
\\
\quad{\bf Bounded time}.  For $t$ bounded, we can use the
medium-$\lambda$ bounds $|G_\lambda|$, $|G_{\lambda_x}|$,
$|G_{\lambda_y}| \le C$ to obtain
$|\int_{\Gamma_1} e^{\lambda t} G_\lambda d \lambda| \le C_2 |\Gamma_1|$.
This contribution is order $Ce^{-\eta t}$ for bounded time, hence
can be absorbed.\\
\\
{\bf Large time}. For $t$ large, we must instead estimate
$\int_{\Gamma_1} e^{\lambda t} G_\lambda d \lambda$ using the
small-$|\lambda|$ expansions. First, observe that, all coefficient
functions $d_{jk}(\lambda)$ are uniformly bounded (since $|\lambda|$
is bounded in this case).


Expanding $G=\int_{\Gamma}e^{\lambda t}G_\lambda(x,y)d\lambda$ as
$$
\begin{pmatrix} G& G_x \\ G_y & G_{xy} \end{pmatrix}= \int_{\Gamma}
e^{\lambda t} \begin{pmatrix} G_\lambda & G_{\lambda_x} \\
G_{\lambda_y} & G_{\lambda_{xy}} \end{pmatrix} \, d\lambda
$$
we estimate the $\int_{\Gamma_1}$ contributions to $G$, $G_x$ and
$G_y$ simultaneously.
\\
{\bf Case II(i). $(0<y<x)$}.  By our low-frequency estimates, we have
\begin{equation} \label{gamma}
\begin{aligned}
\int_{\Gamma} e^{\lambda t} \begin{pmatrix} G_\lambda & G_{\lambda_x} \\
G_{\lambda_y} & G_{\lambda_{xy}} \end{pmatrix} \, d\lambda &=
\int_{\Gamma} \sum_{j,k} e^{\lambda t} \phi^+_j (x) d_{jk} \tilde
\psi^+_k (y) d\lambda\\
&\quad +\int_{\Gamma} \sum_{j,k} e^{\lambda t} \psi^+_k (x)  \tilde
\psi^+_k (y) d\lambda,\\
\end{aligned}
\end{equation}
where each $d_{jk}$ is analytic, hence bounded. We estimate
separately each of the terms
$$
\int_{\Gamma_1} e^{\lambda t} \phi^+_j (x) d_{jk} \tilde \psi^+_k
(y) d\lambda
$$
on the righthand side of \eqref{gamma}.
Estimates for terms
$$
\int_{\Gamma} \sum_{j,k} e^{\lambda t} \psi^+_k (x)  \tilde \psi^+_k (y) d\lambda
$$
go similarly.
\medskip

{\bf Case II(ia).} First, consider the critical case $a^+_j
>0$, $a^+_k
<0$ . For this case,
$$|\phi^+_{j(x)} d_{jk} \tilde \psi^+_k (y)| \le C e^{Re(\rho^+_j x - \nu^+_k y)},$$
where
$$
\begin{cases} \nu^+_k (\lambda) = -\lambda/a^+_k + \lambda^2 \beta^+_k
/ (a^+_k)^3 + \CalO (\lambda^3) \cr \rho^+_j (\lambda) = -\lambda /
a^+_j + \lambda^2 \beta^+_j / (a^+_j)^3 + \CalO(\lambda^3).
\end{cases}
$$

Set
$$
\bar{\alpha}= \frac{a^+_k x/a^+_j -y - a^+_k t}{2t}, \quad
p:=\frac{\beta^+_j a^+_k x / (a^+_j)^3 - \beta^+_k y / (a^+_k)^2
}{t}
>0.
$$
Define $\Gamma_{1a}'$ to  be the portion contained in
$\Omega_\theta$ of the hyperbola
\begin{equation}
\begin{aligned} \label{rho-nu}
&Re(\rho_j^+x - \nu_k^+y) + \CalO(\lambda ^3)(|x|+|y|)\cr &\quad=
(1/a_k^+)Re[\lambda(-a^+_k x/ a^+_j + y) + \lambda^2 (x \beta^+_j
a^+_k / (a^+_j)^3 - y \beta^-_k / (a^+_k)^2 )] \cr &\quad \equiv
\const \cr {} &\quad= (1/a_k^-) [(\lambda_{min}(-a^-_k x/ a^+_j +
y) + \lambda^2_{min} (x \beta^+_j a^+_k / (a^+_j)^3 - y \beta^+_k /
(a^+_k)^2)], \cr
\end{aligned}
\end{equation}
where
\begin{equation}
\lambda_{min} :=
\begin{cases}
\frac{\bar{\alpha}}{p} &  if  \quad |\frac{\bar{\alpha}}{p}|\le
\epsilon \cr \pm \epsilon & if \quad \frac{\bar{\alpha}}{p} \gtrless
\epsilon
\end{cases}
\end{equation}

Denoting by $\lambda_1$, $\lambda^*_1$, the intersections of this
hyperbola with $\partial \Omega_\theta$, define $\Gamma_{1_b}'$ to
be the union of $\lambda_1 \lambda_0$ and $\lambda^*_0 \lambda^*_1$,
and define $\Gamma_1' = \Gamma_{1_a}' \cup \Gamma_{1_b}'$. Note that
$\lambda = \bar{\alpha}/p$ minimizes the left
hand side of \eqref{rho-nu} for $\lambda$ real. Note also that
that $p$ is bounded for $\bar{\alpha}$
sufficiently small, since $\bar{\alpha}\le \epsilon$ implies that
$$
(|a_k^+x/ a_j^+| + |y|)/t \le 2|a_k^+| + 2\epsilon
$$
i.e. $(|x|+|y|)/t$ is controlled by $\bar{\alpha}$.

With these definitions, we readily obtain that
\begin{equation}
\begin{aligned}
Re(\lambda t + \rho^+_j x - \nu^+_k y) &\le -(t/a^-_k)
(\bar{\alpha}^2/4p) - \eta Im (\lambda)^2 t \cr {} &\le -
\bar{\alpha}^2 t/M - \eta Im (\lambda)^2 t,
\end{aligned}
\end{equation}
for $\lambda \in \Gamma_{1_a}'$ (note: here, we have used the
crucial fact that $\bar{\alpha}$ controls $(|x|+|y|)/t$, in bounding
the error term $\CalO(\lambda^3)(|x|+|y|)/t$ arising from expansion
Likewise, we obtain for any $q$ that
\begin{equation} \label{gamma'}
\int_{\Gamma_{1_a}'} |\lambda|^q e^{Re(\lambda t + \rho^+_j x -
\nu^-_k y)} d\lambda \le C t^{-\frac{1}{2} -\frac{q}{2}}
e^{-\bar{\alpha}^2 t / M},
\end{equation}
for suitably large $C,\, M>0$ (depending on $q$). Observing that
$$
\bar{\alpha}= (a_k^+/a_j^+)(x- a_j^+ (t - |y/a_k^+|))/2t,
$$
we find that the contribution of \eqref{gamma'} can be absorbed in
the described bounds for $t\ge |y/a_k^-|$. At
the same time, we find that $\bar{\alpha}\ge x > 0$ for $t\le
|y/a_k^+|$, whence
$$
\bar{\alpha}\ge (x- y - a_j^+t)/Mt +|x|/M,
$$
for some $\epsilon>0$ sufficiently small and  $M>0$ sufficiently
large.

This gives
$$
e^{-\bar{\alpha}^2/p}\le e^{- (x-y- a_k^+t)^2/Mt} e^{-\eta |x|}
$$
provided $|x|/t >a_j^+$, a contribution which can again be absorbed.
On the other hand, if $t\le |x/a_j^+|$, we
can use the dual estimate
\begin{equation}
\begin{aligned}
\bar{\alpha}&= (-y- a_k^+(t - |x/a_j^+|))/2t \cr &\ge  (x-y-
a_k^+t)/Mt  + |y|/M,
\end{aligned}
\end{equation}
together with $|y|\ge |a_k^- t|$, to obtain
$$
e^{-\bar{\alpha}^2/p}\le e^{- (x-y- a_j^+t)^2/Mt} e^{-\eta |y|},
$$
a contribution that can likewise be absorbed.


{\bf Case II(ib).}
In case $a^+_j<0$ or $a^{+}_k >0$, terms $|\varphi_j^+| \le C
e^{-\eta |x|}$ and $|\tilde{\psi}_j^+| \le C e^{-\eta |y|}$ are strictly
smaller than those already treated in  Case II(ia), so may be
absorbed in previous terms.
\bigskip
\\
{\bf Case II(ii) $(0<x<y)$.} The case $0<x<y$ can be treated very
similarly to the previous one; see \cite{ZH} for details.
This completes the proof of Case II, and the theorem.
\section{Nonlinear Analysis}
Introducing the perturbation variable
\begin{equation}
u(x,t):=\tilde u(x,t)-\bar u(x),
\end{equation}
we obtain
\begin{equation}
\label{perteq} u_t-Lu=Q(u)_x,
\end{equation}
where the second-order Taylor remainder satisfies
\begin{equation}
\begin{aligned}
Q(u)&:=
f(\bar u+u)- f(\bar u)- df(\bar u)u
=\mathcal{O}(|u|^2 )\\
\end{aligned}
\end{equation}
so long as $|u|$ remains bounded.

\begin{lem}[Integral formulation]\label{shorttime}
Under the assumptions of Theorem \ref{nonlin},
there exists a classical solution of \eqref{perteq} for
$0< t\le T$, $T>0$, continuous in $L^\infty(x)$ at $t=0$,
extending for all $t>0$ such that
$u(\cdot, t)$ remains sufficiently small in $L^1\cap L^\infty$,
given by
\begin{equation}\label{u}
\begin{aligned}
  u(x,t)&=\int^\infty_{0}G(x,t;y)g(y)\,dy
  +\int^t_0  G_y(x,t-s;0)B h(s)\, ds\\
  &\quad -\int^t_0 \int^\infty_{0} G_y(x,t-s;y)
  Q(u)(y,s)\,dy\,ds.\\
\end{aligned}
\end{equation}
\end{lem}

\begin{proof}
From Lemma \eqref{greendual} and the inverse Laplace representation
\eqref{iLT}
we find that $G(x,t-s;y)$
considered as a function of $y,s$ satisfies the adjoint equation
\begin{equation}\label{timeadj}
(\partial_s- L_y)^*G^*(x, t- \cdot; \cdot)= 0,
\end{equation}
or
\begin{equation}\label{explicit}
-G_s -(GA)_y + GA_y= (G_yB)_y.
\end{equation}
Likewise, reviewing the construction of the resolvent,
we find $G_\lambda(x,0)\equiv 0$, yielding
\begin{equation}
G(x,t-s;0)\equiv 0.
\end{equation}
That is, $G^*(x, t-\cdot; \cdot)$ is the Green function for the
adjoint equation, as may alternatively be seen directly
by a duality argument analogous to the proof of Lemma \eqref{greendual}.

Thus, integrating $G$ against \eqref{perteq}, integrating
by parts, and using the fact that $G=0$ and $u=h$ on the boundary
$y=0$, we obtain for any classical solution of \eqref{perteq} that
\begin{equation}
\begin{aligned}
\int_0^t\int_0^\infty & G(x,t-s;y)Q(u(y,s))_y \, dy\, ds=\\
& \int_0^t\int_0^\infty G(x,t-s;y)(\partial_s-L_y)u(y,s) \, dy\, ds\\
&=
\int_0^t\int_0^\infty ((\partial_s-L_y)^* G^*)^*(x,t-s;y)u(y,s) \, dy\, ds\\
  &+ u(x,t) -\int^\infty_{0}G(x,t;y)g(y)\,dy
  -\int^t_0  G_y(x,t-s;0)B h(s)\, ds,\\
\end{aligned}
\end{equation}
from which we obtain \eqref{u} by rearranging and integrating
by parts the term
$\int_0^t\int_0^\infty  G(x,t-s;y)Q(u(y,s))_y \, dy\, ds$.

Indeed, \eqref{u} may be taken as the definition of a weak solution
in $L^\infty(x,t)$.
(One can see using convolution identities
that this agrees with the usual definition in terms of
integration against test functions $\phi\in C^\infty_0(\R\times \R)$.)
Existence of weak solutions can be obtained by a standard contraction
mapping/continuation argument using the convolution bounds of
Lemmas \ref{iniconvolutions}--\ref{boundaryconvolution} below;
we omit the details, since we shall carry out quite similar but
more difficult estimates in the proof of stability.
Smoothness of solutions may then be obtained by a bootstrapping
argument as sketched in Appendix \ref{smoothness}.
\end{proof}

To establish stability, we use the following lemmas proved in \cite{HZ}.


\begin{lem}[Linear estimates \cite{HZ}]\label{iniconvolutions}
Under the assumptions of Theorem
\ref{nonlin},
\begin{equation}\label{iniconeq}
\begin{aligned}
\int_{0}^{+\infty}|G(x,t;y)|(1+|y|)^{-3/2}\, dy
&\le C(\theta+\psi_1+\psi_2)(x,t),\\
\end{aligned}
\end{equation}
for $0\le t\le +\infty$, some $C>0$.
\end{lem}
\begin{lem}[Nonlinear estimates \cite{HZ}]\label{convolutions}
Under the assumptions of Theorem \ref{nonlin},
\begin{equation}\label{coneq}
\begin{aligned}
\int_0^t\int_{0}^{+\infty}|G_y(x,t-s;y)|\Psi(y,s)\, dy
ds
&\le C(\theta+\psi_1+\psi_2)(x,t),\\
\end{aligned}
\end{equation}
for $0\le t\le +\infty$, some $C>0$, where
\begin{equation}\label{source}
\begin{aligned}
\Psi(y,s)&:=
(1+s)^{1/2}s^{-1/2}(\theta + \psi_1+\psi_2)^2(y,s)\\
&\qquad + (1+s)^{-1} (\theta+\psi_1+\psi_2)(y,s).
\end{aligned}
\end{equation}
\end{lem}

We require also the following estimate accounting
boundary effects.

\begin{lem}[Boundary estimate]\label{boundaryconvolution}
Under the assumptions of Theorem
\ref{nonlin},
\begin{equation}\label{bconeq}
\begin{aligned}
\Big|\int_{0}^{t}G_y(x,t-s;0)B h(s)\, ds\Big|
&\le CE_0(\theta+\psi_1+\psi_2)(x,t),\\
\end{aligned}
\end{equation}
for $0\le t\le +\infty$, some $C>0$.
\end{lem}

\begin{proof}
The estimate on $\int_0^{t-1}$, where $G_y(x,t-s;0)$
is nonsingular,
follows readily by estimates similar to but somewhat
simpler than those of Lemma \eqref{convolutions},
which we therefore omit.

To bound the singular part $\int_{t-1}^t$, we integrate
\eqref{explicit} in $y$ from $0$ to $+\infty$, recalling
that $G(x,t-s;0)\equiv 0$, to obtain
\begin{equation}\label{keyrel}
G_yB= -\int_0^{+\infty}A_y(y)G(x,t-s;y)\, dy -
\int_0^{+\infty}G_s(x,t-s;y)\, dy.
\end{equation}
Substituting in the lefthand side of \eqref{bconeq}, and
integrating by parts in $s$, we obtain
\begin{equation}\label{partsest}
\begin{aligned}
\int_{t-1}^t G_yBh(s)\, ds
&= \int_{0}^1 \Big(\int_0^{+\infty}A_y(y)G(x,\tau ;y)\, dy \Big) h(t-\tau)\, d\tau  \\
&\quad - \int_0^1 \Big( \int_0^{+\infty}G(x,\tau;y)\, dy \Big)
h'(t-\tau)\, d\tau\\
&\quad +
\Big(\int_0^{+\infty}G(x,1;y)\, dy\Big) h(t-1),
\end{aligned}
\end{equation}
which by $\int |G|  dy\le C$
has norm bounded by $\max_{0\le \tau\le 1}(|h|+|h'|)(t-\tau)$.

Combining this with the more straightforward estimate
\begin{equation}\label{basicest}
\begin{aligned}
\Big|\int_{t-1}^{t}G_y(x,t;0)B h(s)\, ds\Big|&\le
\int_{0}^{1}|G_y(x,\tau ;0)|B h(s)\, ds\\
&\le
C\max_{0\le \tau\le 1}|h(t-\tau)|
\int_{0}^{1}\tau^{-1} e^{-|x|^2/C\tau}\, d\tau\\
&=C|x|^{-2}\max_{0\le \tau\le 1}|h(t-\tau)|\\
&\quad \times
\int_{0}^{1}(|x|^2/\tau) e^{-|x|^2/C\tau}\, d\tau\\
&\le C\max_{0\le \tau\le 1}|h(t-\tau)| |x|^{-2},
\end{aligned}
\end{equation}
we find that the contribution from $\int_{t-1}^t$
has norm bounded by
$$
\max_{0\le \tau\le 1}(|h|+|h'|)(t-\tau)(1+|x|)^{-2}.
$$
Combining this estimate with the one for $\int_0^{t-1}$,
we obtain \eqref{bconeq}.
\end{proof}

With these preparations, the proof of
stability is straightforward.

\begin{proof}[Proof of Theorem \ref{nonlin}]
Define
\begin{equation}\label{zeta2}
 \zeta(t):= \sup_{y, 0\le s \le t}
 |u|(\theta+\psi_1+\psi_2)^{-1}(y,t).
\end{equation}
We will establish:

{\it Claim.} For all $t\ge 0$ for which a
solution exists with $\zeta$ uniformly bounded by some fixed,
sufficiently small constant, there holds
\begin{equation}
\label{claim} \zeta(t) \leq C_2(E_0 + \zeta(t)^2).
\end{equation}
\medskip
{}From this result, provided $E_0 < 1/4C_2^2$, we have that
$\zeta(t)\le 2C_2E_0$ implies $\zeta(t)< 2C_2E_0$, and so we may
conclude by continuous induction that
 \begin{equation}
 \label{bd}
  \zeta(t) < 2C_2E_0
 \end{equation}
for all $t\geq 0$. (By Lemma \ref{shorttime} and standard short-time
estimates, $u\in C^0(x)$ exists and
$\zeta$ remains continuous so long as $\zeta$ remains bounded by
some uniform constant, hence \eqref{bd} is an open condition.
From \eqref{bd} and the definition of $\zeta$
in \eqref{zeta2} we then obtain the bounds of \eqref{pointwise}.
Thus, it remains only to establish the claim above.
\medskip

{\it Proof of Claim.} We must show that
$u(\theta+\psi_1+\psi_2)^{-1}$ is
bounded by $C(E_0 + \zeta(t)^2)$, for some $C>0$, all $0\le s\le t$,
so long as $\zeta$ remains sufficiently small.
By \eqref{zeta2},
we have for all $t\ge 0$ and some $C>0$ that
\begin{equation}\label{ubounds}
\begin{aligned}
|u(x,t)| &\le \zeta (t)(\theta +\psi_1+\psi_2)(x,t),\\
\end{aligned}
\end{equation}
and therefore
\begin{equation}\label{Nbounds}
\begin{aligned}
|Q(u)(y,s)|&\le C\zeta(t)^2 \Psi(y,s)
\end{aligned}
\end{equation}
with $\Psi$ as defined in \eqref{source}, for $0\le s\le t$.
Combining \eqref{Nbounds} with representation
(\ref{u}) and applying Lemmas
\ref{iniconvolutions}--\ref{boundaryconvolution}, we obtain
\begin{equation}
 \begin{aligned}
  |u(x,t)| &\le
  \int^\infty_{0} |\tilde G(x,t;y)| |g(y)|\,dy
  +\Big|\int^t_0  G_y(x,t-s;0)B h(s)\, ds\Big|
   \\
 &\qquad +\int^t_0
  \int^\infty_{0}|\tilde G_y(x,t-s;y)||(Q(u))(y,s)| dy \, ds \\
  & \le
  E_0 \int^\infty_{0} |\tilde G(x,t;y)|(1+|y|)^{-3/2}\,dy
   \\
  &\quad +\Big|\int^t_0  G_y(x,t-s;0)B h(s)\, ds\Big|\\
 &\quad +
C\zeta(t)^2 \int^t_0
  \int^\infty_{0}|\tilde G_y(x,t-s;y)|
\Psi(y,s) dy \, ds \\
&\le C(E_0+\zeta(t)^2)(\theta + \psi_1+\psi_2)(x,t).
\end{aligned}
\end{equation}

Dividing by $(\theta+\psi_1+\psi_2)(x,t)$,
we obtain \eqref{claim} as claimed.
This completes the proof of the claim, and the theorem.
\end{proof}

\appendix

\section{Smoothness of solutions}\label{smoothness}

In this appendix, we briefly sketch the proof that weak solutions
defined by \eqref{u} are necessarily smooth, classical solutions as well,
by indicating how to get the necessary derivative bounds.

\medskip
{\bf Time-derivative.}
Rewriting the second, boundary term, on the righthand side of \eqref{u}
using its convolution structure, as
$$
 \int^t_0  G_y(x,\tau ;0)B h(t-\tau)\, d\tau,
$$
and differentiating in $t$, we obtain
$$
 G_y(x,t ;0)B h(0) +
 \int^t_0  G_y(x,\tau ;0)B h'(t-\tau)\, d\tau,
$$
for which the first term is bounded and smooth for $x,t>0$, and the
second by the same estimate as in \eqref{basicest} is bounded by
$$
C|x|^{-2}\int_0^t |h'(t-\tau)|\,d\tau \le C|x|^{-2}\log (1+t).
$$

Differentiating the first and third terms with respect to $t$
and integrating the third term by parts in $y$ yields
\begin{equation}
\begin{aligned}
  \int^\infty_{0}G(x,t;y)g(y)\,dy
  &-\int^t_{t/2} \int^\infty_{0} G_y(x,t-s;y) Q(u)_s(y,s)\,dy\,ds\\
  &\quad -\int^{t/2}_0 \int^\infty_{0} G_{yt}(x,t-s;y) Q(u)(y,s)\,dy\,ds,
\end{aligned}
\end{equation}
from which, in combination with the boundary estimate already performed,
we may readily obtain a short-time bound $|u_t|\le C|x|^{-2}t^{-1}$
by Picard iteration.

{\bf Spatial-derivatives.}  Likewise, differentiating \eqref{partsest} with respect
to $x$, we may bound the $x$-derivative of
the boundary term $\int_0^t G_yB\, ds$ by
$$
C\int_0^t \tau^{-1/2} (|h'|+|h|)(t-\tau)|\,d\tau \le C\log (1+t).
$$
Differentiating the first and third terms of the righthand side of \eqref{u}
with respect to $x$ and integrating the third term by parts in $y$ yields
\begin{equation}
\begin{aligned}
  \int^\infty_{0}G(x,t;y)g(y)\,dy
  &-\int^t_{t/2} \int^\infty_{0} G_x(x,t-s;y) Q(u)_y(y,s)\,dy\,ds\\
  &\quad -\int^{t/2}_0 \int^\infty_{0} G_{yx}(x,t-s;y) Q(u)(y,s)\,dy\,ds,
\end{aligned}
\end{equation}
from which, in combination with the boundary estimate already performed,
we obtain a short-time bound $|u_x|\le Ct^{-1/2}$ by Picard iteration.
From the bounds on $|u_t|$ and $|u_x|$, finally, we obtain bounds on
$|u_{xx}|$ by the equation satisfied by $u$.

\medskip
{\bf Acknowledgement.}
This work was supported in part by the
National Science Foundation grant number DMS-0300487.


\end{document}